\documentclass[12pt]{article}

\usepackage[a4paper,margin=1in]{geometry}
\usepackage[T1]{fontenc}
\usepackage[utf8]{inputenc}
\usepackage{lmodern}
\usepackage{amsmath,amssymb,amsthm,amsfonts,mathtools}
\usepackage{mathrsfs}
\usepackage{array,booktabs,multirow}
\usepackage{graphicx}
\usepackage{tikz,cite}
\usetikzlibrary{arrows.meta,positioning,calc,fit,shapes.misc}
\usepackage{algorithm}
\usepackage{algpseudocode}
\usepackage{hyperref}
\usepackage{enumitem}
\usepackage{float}
\usepackage{caption}

\everymath{\displaystyle}
\renewenvironment{proof}{{\bf \noindent Proof.}}{\qed}
\usetikzlibrary{matrix, calc, positioning}

\hypersetup{
	colorlinks=true,
	linkcolor=blue,
	citecolor=blue,
	urlcolor=blue
}

\newtheorem{theorem}{Theorem}[section]
\newtheorem{proposition}[theorem]{Proposition}
\newtheorem{lemma}[theorem]{Lemma}
\newtheorem{corollary}[theorem]{Corollary}

\newtheorem{openproblem}[theorem]{Open Problem}
\theoremstyle{definition}
\newtheorem{definition}[theorem]{Definition}
\newtheorem{example}[theorem]{Example}
\newtheorem{remark}[theorem]{Remark}

\newcommand{\Z}{\mathbb{Z}}
\newcommand{\F}{\mathbb{F}}
\newcommand{\LL}{\mathbb{L}}
\newcommand{\PP}{\mathbb{P}}
\newcommand{\calL}{\mathcal{L}}
\newcommand{\calT}{\mathcal{T}}
\newcommand{\calD}{\mathcal{D}}
\newcommand{\calO}{\mathcal{O}}
\newcommand{\calC}{\mathcal{C}}
\newcommand{\one}{\mathbf{1}}
\newcommand{\I}{\mathbf{I}}
\newcommand{\J}{\mathbf{J}}
\renewcommand{\O}{\mathbf{O}}
\newcommand{\rank}{\operatorname{rank}}
\newcommand{\nullity}{\operatorname{nullity}}
\newcommand{\Ker}{\operatorname{Ker}}
\newcommand{\Image}{\operatorname{Im}}
\newcommand{\tr}{\operatorname{tr}}
\newcommand{\diam}{\operatorname{diam}}
\newcommand{\girth}{\operatorname{girth}}

\title{Spectrum of zero-divisor graphs of Lipschitz quaternion rings modulo $n$}
\author{{\small Bilal Ahmad Rather}\\[2mm]
	{\small School of Mathematics and Statistics, Shandong University of Technology, Zibo 255049, China}\\
	\texttt{bilahamadrr@gmail.com}
}
\date{}

\begin{document}
	\pagestyle{myheadings} \markboth{Bilal Ahmad Rather}{zero-divisor graphs of Lipschitz quaternion rings}
	\maketitle
	
	\begin{abstract}
		We study the adjacency spectra of zero-divisor graphs associated with Lipschitz quaternion rings modulo $n$. For an odd prime $p$, the identification $\mathbb{L}_p\cong M_2(\mathbb{F}_p)$ gives a kernel--image partition of the nonzero singular matrices and an explicit block representation of the adjacency matrix. This structure leads to a spectral decomposition, closed expressions for the spectral radius, and information on rank, nullity, eigenvalue multiplicities, and graph energy. For powers of $2$, we determine the graph at $n=2$ and use square-zero ideals in $\mathbb{L}_{2^t}$ to construct large complete subgraphs and obtain corresponding spectral and energy bounds. The resulting framework provides a unified structural approach to the odd-prime and two-adic cases and reduces large adjacency-matrix computations to substantially smaller algebraic models.
	\end{abstract}

	
	
	\section{Introduction}
	
	Graphs attached to rings provide a useful way to translate algebraic annihilation relations into combinatorial structure. Beck's work initiated this viewpoint \cite{beck1988}, Anderson and Livingston introduced the now-standard zero-divisor graph for commutative rings \cite{anderson1999}, and Redmond extended the construction to noncommutative rings \cite{redmond2002}. Matrix rings and quaternion rings are natural test cases because their multiplication carries substantial internal structure that can be reflected in the graph.
	Zero-divisor graphs have since been studied through connectivity, diameter, girth, domination, direct products, and several related invariants \cite{akbari2006,akbari2007,axtell2006,bozic2009,miguel2013,wu2005}. Here we focus on a complementary question: what additional information becomes visible when the graph is treated through its adjacency matrix?
	
	The adjacency matrix records walk counts and supports spectral, rank, nullity, energy, and equitable-partition methods; see \cite{cvetkovic1995,godsil2001,brouwer2012,horn2013}.
	For graph families coming from finite rings, repeated rows, block patterns, and equitable partitions can compress a large adjacency problem to a much smaller matrix. This is particularly useful in noncommutative examples, where brute-force multiplication becomes expensive as the ring grows. Related constructions for Gaussian integers modulo $n$ and other zero-divisor graphs appear in \cite{abuosba2008,abuosba2011,bilal}.
	
	Among the most natural noncommutative examples are the Lipschitz quaternion rings modulo $n$. Consider
	$$
	\LL_n=\Z_n[i,j,k]
	=\{a+bi+cj+dk:\ a,b,c,d\in \Z_n\},
	$$
	with the usual quaternion relations inherited modulo $n$. Grau, Miguel, and Oller-Marc\'en studied the zero-divisor graphs of $\LL_n$ and obtained formulas for several graph invariants, including the number of vertices, diameter, girth, and domination number \cite{grau2017}. Their work also records the structural distinction between odd moduli and powers of $2$. Related zero-divisor graphs of matrix rings and other finite rings have been investigated in \cite{akbari2006,akbari2007,bozic2009,miguel2013,wu2005}. The aim here is more specific: we use the adjacency matrix to expose the repeated block structure created by the ring geometry and then derive spectral consequences from that structure.
	
	For an odd prime $p$, the isomorphism $\LL_p\cong M_2(\F_p)$ classifies every nonzero zero divisor by the pair consisting of its kernel and image, both of which are projective lines in $\F_p^2$. The resulting type partition turns adjacency into an incidence condition on $\PP^1(\F_p)$. We obtain an explicit Kronecker block model, exact formulas for the two degrees and the edge count, and exact clique and independence numbers. The block model first isolates large forced $0$- and $-1$-eigenspaces and leaves a reduced matrix $B_p$ of order $(p+1)^2$. A second symmetry decomposition of this reduced space diagonalizes $B_p$ completely, so the full adjacency spectrum, nullity, rank, spectral radius, and energy are all explicit.
	
	For powers of $2$, the same matrix-ring reduction is unavailable. We first determine the graph at $n=2$ and identify it with the friendship graph $F_3$. For $t\ge2$, we use the square-zero ideal $2^{\lceil t/2\rceil}\LL_{2^t}$ to produce a complete principal block in the adjacency matrix. Its size is explicit, which gives quantitative lower bounds for the clique number, spectral radius, number of edges, negative inertia, rank, and graph energy. We also exploit the nested ideals $2^a\LL_{2^t}$ and $2^b\LL_{2^t}$ when $a+b\ge t$; suitable disjoint layers are completely joined and yield large complete bipartite subgraphs. In particular, for odd $t$ this improves the spectral-radius bound obtained from the square-zero clique alone.
	
	For clarity, we collect here the symbols that occur before their formal development later in the paper. We write $\Z_n=\mathbb Z/n\mathbb Z$, and, for a prime $p$, $\F_p$ is the field with $p$ elements. The ring of Lipschitz quaternions modulo $n$ is $\LL_n$, while $G_n=\Gamma(\LL_n)$ denotes its undirected zero-divisor graph and $A_n=A(G_n)$ its adjacency matrix. The quantities $\rho(A_n)$ and $\nullity(A_n)$ denote the spectral radius and nullity. For odd $p$, $\PP^1(\F_p)$ is the projective line of one-dimensional subspaces of $\F_p^2$, $\calL$ is its set of lines, and $\calT=\calL\times\calL$ is the set of kernel--image types. The matrix $H_p$ records incidence between such ordered pairs, whereas $D_p$ marks the diagonal types $(L,L)$. As usual, $\J_m$ and $\I_m$ are the all-ones and identity matrices of order $m$, $\otimes$ denotes the Kronecker product, and $\sim$ denotes permutation similarity. Finally, $B_p=(p-1)H_p-D_p$ is the reduced type-space matrix obtained by restricting $A_p$ to vectors that are constant on each type class.
	
	The problem considered throughout is therefore the following. For $G_n=\Gamma(\LL_n)$, determine structural information about $A_n=A(G_n)$ that can be read from the ring and use it to compute or bound spectral invariants such as $\rho(A_n)$, $\nullity(A_n)$, special eigenvalue multiplicities, and energy. For odd primes we prove
	$$
	A_p\sim H_p\otimes \J_{p-1}-D_p\otimes \I_{p-1},
	$$
	where $H_p$ records incidence among ordered pairs of projective lines and $D_p$ records the diagonal types. The decomposition compresses the adjacency problem before any dense matrix is formed.
	
	\medskip
	The article is organized as follows: Section~\ref{sec:preliminaries} fixes notation and recalls the structural facts used later. Section~\ref{sec:oddprime} develops the projective type partition, the block form of $A_p$, and exact clique and independence numbers. Section~\ref{sec:spectral} derives the invariant-subspace decomposition, completely diagonalizes the reduced matrix $B_p$, and obtains exact multiplicities, rank, and spectral radius. Section~\ref{sec:twopower} treats the exact case $n=2$, the square-zero clique, annihilating ideal layers, and inertia bounds for $2^t$. Section~\ref{sec:algorithm} separates the compressed type-incidence construction from the cost of materializing the full dense matrix and records numerical examples. Section~\ref{sec:energy} develops the energy decomposition, proves an exact odd-prime formula and hyperenergeticity, and retains quantitative lower bounds useful for comparison. Section~\ref{con} summarizes the results and lists several open problems.
	
	\section{Preliminaries and notation}\label{sec:preliminaries}
	
	We recall the basic notation used throughout the paper. For $n\ge 2$, let
	$$
	\LL_n=\Z_n[i,j,k]
	=\{a+bi+cj+dk:\ a,b,c,d\in \Z_n\}
	$$
	be the ring of Lipschitz quaternions modulo $n$.
	Its undirected zero-divisor graph $G_n=\Gamma(\LL_n)$ is the simple graph whose vertices are the nonzero zero divisors of $\LL_n$, with distinct vertices $x,y$ adjacent if and only if $xy=0$ or $yx=0$. In this paper, a nonzero element $x$ of a possibly noncommutative ring is called a zero divisor if $xy=0$ or $yx=0$ for some nonzero $y$. Since $G_n$ is finite and simple, its adjacency matrix $A_n=A(G_n)$ is a symmetric $(0,1)$ matrix with zero diagonal. We write $\rho(A_n)$ for the spectral radius, $\nullity(A_n)$ for the nullity, and $\chi_{A_n}(\lambda)=\det(\lambda I-A_n)$ for the characteristic polynomial.
	
	For a graph $G$ with adjacency matrix $A(G)$, the \emph{nullity} of $G$ is $\nullity(A(G))$, and the \emph{rank} is $\rank(A(G))$. A partition $V(G)=V_1\sqcup\cdots\sqcup V_r$ is called \emph{equitable} if every vertex in $V_i$ has the same number of neighbors in $V_j$ for each pair $i,j$. The associated $r\times r$ quotient matrix will be denoted by $Q$, see \cite{godsil2001,brouwer2012}. 
	
	\medskip
	The following facts are standard and will be used repeatedly.
	\begin{proposition}[\cite{lam2001,grau2015,grau2017}]\label{prop:finitefacts}
		Let $R$ be a finite ring with identity. Then the following holds.
		\begin{enumerate}[label=\textup{(\roman*)},noitemsep]
			\item Every nonunit of $R$ is a zero divisor.
			\item If $R=R_1\oplus R_2$, then a nonzero element $(a_1,a_2)$ is a zero divisor if and only if at least one coordinate is a nonunit; equivalently, at least one coordinate is either $0$ or a nonzero zero divisor in its factor ring.
			\item If $p$ is an odd prime, then $\LL_p\cong M_2(\F_p)$.
		\end{enumerate}
	\end{proposition}

	\medskip
	The next theorem summarizes graph invariants already known for $G_n$.
	\begin{theorem}[\cite{grau2017}]\label{thm:known}
		Let $n\ge 2$. Then the following holds.
		\begin{enumerate}[label=\textup{(\roman*)},noitemsep]
			\item If $p$ is an odd prime, then $|V(G_p)|=p^3+p^2-p-1,$ and if $t\ge 1$, then
			$|V(G_{2^t})|=2^{4t-1}-1.$ 
			\item If $n$ is a prime power, then $\diam(G_n)=2$, otherwise $\diam(G_n)=3$.
			\item For every $n\ge 2$, we have  $\girth(G_n)=3$.
			\item If $t\ge 1$, then the domination number of $G_{2^t}$ is $1$. If $p$ is an odd prime, then the domination number of $G_p$ is $p+1$.
		\end{enumerate}
	\end{theorem}
	
	\medskip
	For later use, we also record a matrix-ring characterization of adjacency at prime modulus.
	\begin{lemma}\label{lem:rankoneadj}
		Let $p$ be an odd prime and identify $\LL_p$ with $M_2(\F_p)$. If $A,B\in M_2(\F_p)$ are nonzero singular matrices, then
		$AB=0 \iff \Image(B)\subseteq \Ker(A),$ and $BA=0 \iff \Image(A)\subseteq \Ker(B).$ As $A$ and $B$ have rank $1$, this is equivalent to $AB=0 \iff \Image(B)=\Ker(A),$ and $BA=0 \iff \Image(A)=\Ker(B).$
	\end{lemma}
	
	\begin{proof}
		The implications $AB=0\iff \Image(B)\subseteq \Ker(A)$ and $BA=0\iff \Image(A)\subseteq \Ker(B)$ hold for arbitrary linear maps. For nonzero singular $2\times 2$ matrices over the field $\F_p$, rank-nullity gives $\dim\Ker(A)=\dim\Image(A)=1$, and similarly for $B$. Hence, the relevant inclusions are equivalent to equality of lines.
	\end{proof}
	
	\medskip
	Theorem~\ref{thm:known} supplies the graph-theoretic background. The next sections focus instead on adjacency-matrix structure. In the odd-prime case, the kernel--image geometry of rank-one matrices gives a natural compressed model, and the point of the analysis is to make the resulting block structure and its spectral consequences explicit.
	
	\section{Odd prime modulus: a block decomposition of the adjacency matrix}\label{sec:oddprime}
	
	Throughout this section, $p$ denotes an odd prime, and from Proposition~\ref{prop:finitefacts}, we use the identification $\LL_p\cong M_2(\F_p)$. Let $V=\F_p^2$, and let $\calL$ be the set of $1$-dimensional subspaces of $V$. Since the projective line has cardinality $|\PP^1(\F_p)|=p+1$, so we have $|\calL|=p+1$. For a nonzero singular matrix $A\in M_2(\F_p)$, both $\Ker(A)$ and $\Image(A)$ lie in $\calL$. This motivates the following classification.
	
	\begin{definition}
		For $L,M\in \calL$, let
		$$
		\calC_{L,M}:=\{A\in M_2(\F_p): A\neq 0,\ \det(A)=0,\ \Ker(A)=L,\ \Image(A)=M\}.
		$$
		We call $(L,M)$ the \emph{type} of $A$, and the family $\{\calC_{L,M}\}_{L,M\in\calL}$ the \emph{type partition}.
	\end{definition}
	
	\medskip
	The following lemma gives the order of $\calC_{L,M}.$
	\begin{lemma}\label{lem:classsize}
		For every $L,M\in\calL$, we have $|\calC_{L,M}|=p-1.$ Consequently, the type partition has $(p+1)^2$ parts, each of cardinality $p-1$, and $|V(G_p)|=(p+1)^2(p-1)=p^3+p^2-p-1.$ 
	\end{lemma}
	
	\begin{proof}
		For $L,M\in \calL$, a nonzero linear transformation $A:V\to V$ with $\Ker(A)=L$ and $\Image(A)=M$ factors uniquely as
		$$
		V \longrightarrow V/L \xrightarrow{\ \varphi\ } M \hookrightarrow V,
		$$
		where $\varphi:V/L\to M$ is a nonzero linear map between $1$-dimensional $\F_p$-vector spaces. Such maps are in bijection with $\F_p^\times$, and hence there are exactly $p-1$ possibilities. Now, the vertex count follows immediately.
	\end{proof}
	
	\medskip
	We now give the first main theorem of the paper.	
	\begin{theorem}\label{thm:block}
		Let $\calT=\calL\times \calL$ be indexed by ordered pairs of lines, and let a $(p+1)^2\times (p+1)^2$ matrix $H_p=(h_{\alpha,\beta})_{\alpha,\beta\in\calT}$ be defined as
		$$
		h_{(L,M),(L',M')}=
		\begin{cases}
			1,& \text{if } M=L' \text{ or } M'=L,\\
			0,& \text{otherwise.}
		\end{cases}
		$$
		Also, let the diagonal matrix $D_p=(d_{\alpha,\beta})_{\alpha,\beta\in\calT}$ be 
		$$
		d_{(L,M),(L',M')}=
		\begin{cases}
			1,& \text{if } (L,M)=(L',M') \text{ and } L=M,\\
			0,& \text{otherwise.}
		\end{cases}
		$$
		Then the adjacency matrix $A_p=A(G_p)$ is permutation similar to
		$$
		A_p \sim H_p\otimes \J_{p-1}-D_p\otimes \I_{p-1}.
		$$
		Equivalently, after ordering the vertices by type classes $\calC_{L,M}$, the block indexed by $(L,M)$ and $(L',M')$ is
		$$
		B_{(L,M),(L',M')}=
		\begin{cases}
			\J_{p-1}-\I_{p-1},& \text{if } (L,M)=(L',M') \text{ and } L=M,\\[1mm]
			\J_{p-1},& \text{if } (L,M)\neq (L',M') \text{ and } (M=L' \text{ or } M'=L),\\[1mm]
			\O_{p-1},& \text{otherwise.}
		\end{cases}
		$$
	\end{theorem}
	\begin{proof}
		Choose an ordering of the type set $\calT$ and, inside each class $\calC_{L,M}$, choose any ordering of its $p-1$ elements. This produces an ordering of all vertices, since the type classes are pairwise disjoint and, by Lemma~\ref{lem:classsize}, each has the same cardinality $p-1$. We determine the block joining two fixed types $\alpha=(L,M)$ and $\beta=(L',M')$.
		
		Let $A\in\calC_{L,M}$ and $C\in\calC_{L',M'}$. Since both matrices have rank $1$, Lemma~\ref{lem:rankoneadj} gives
		$$
		AC=0\iff \Image(C)=\Ker(A)\iff M'=L,
		$$
		and
		$$
		CA=0\iff \Image(A)=\Ker(C)\iff M=L'.
		$$
		Hence, whenever $\alpha\neq\beta$, either every pair in $\calC_{L,M}\times\calC_{L',M'}$ is adjacent or no pair is adjacent. More precisely, the off-diagonal block is $\J_{p-1}$ exactly when $M=L'$ or $M'=L$, and is $\O_{p-1}$ otherwise. Notice that the criterion is symmetric in $\alpha$ and $\beta$, as it must be for the adjacency matrix of the undirected graph.
		
		It remains to determine the block belonging to a single type $\alpha=(L,M)$. For distinct $A,C\in\calC_{L,M}$., the preceding criterion reduces to
		$$
		A\sim C\iff M=L.
		$$
		Thus, if $L\neq M$, no two distinct vertices of $\calC_{L,M}$ are adjacent, and the diagonal block is $\O_{p-1}$. If $L=M$, every two distinct vertices in the class are adjacent, but the graph has no loops; consequently the block is $\J_{p-1}-\I_{p-1}$ rather than $\J_{p-1}$.
		
		Now $H_p$ places a copy of $\J_{p-1}$ in every block for which the incidence condition $M=L'$ or $M'=L$ holds. In particular, its diagonal block is $\J_{p-1}$ precisely for a diagonal type $(L,L)$. The matrix $D_p\otimes\I_{p-1}$ subtracts the identity exactly from those diagonal-type blocks and nowhere else. Therefore the block matrix in the chosen vertex ordering is
		$$
		H_p\otimes\J_{p-1}-D_p\otimes\I_{p-1}.
		$$
		Changing from the original vertex ordering to this type ordering is effected by a permutation matrix. Hence the displayed block matrix is permutation similar to $A_p$, completing the proof.
	\end{proof}
	
	\medskip
	The block form immediately gives the row sums.	
	\begin{corollary}\label{cor:degrees}
		Let $A$ be a vertex of $G_p$. Then the following hold.
		\begin{enumerate}[label=\textup{(\roman*)},noitemsep]
			\item If $\Ker(A)=\Image(A)$, then $\deg(A)=2p(p-1)+(p-2)=2p^2-p-2.$ 
			\item If $\Ker(A)\neq \Image(A)$, then $\deg(A)=(2p+1)(p-1)=2p^2-p-1.$ 
		\end{enumerate}
		In particular, $G_p$ is biregular with the two degree values differing by exactly $1$.
	\end{corollary}
	
	\begin{proof}
		Let $A\in \calC_{L,M}$. Then by Theorem~\ref{thm:block}, the neighboring type classes are exactly those $\calC_{L',M'}$ satisfying $M=L'$ or $M'=L$. If $L\neq M$, the union of these classes consists of
		$$
		\{(M,X): X\in \calL\}\ \cup\ \{(X,L): X\in\calL\},
		$$
		which has $2(p+1)-1=2p+1$ distinct types, none of which is $(L,M)$. Each class contributes $p-1$ vertices, thereby giving
		$\deg(A)=(2p+1)(p-1).$ If $L=M$, then the same union has $2p+1$ types, and it  includes its own class $(L,L)$. The other $2p$ classes contribute $2p(p-1)$ neighbors, while the own class contributes $p-2$ neighbors. Thus, $\deg(A)=2p(p-1)+(p-2).$ 
	\end{proof}
	
	\medskip
	The following consequence gives the edge cardinality of $G_{p}.$
	\begin{corollary}\label{cor:edges}
		The number of edges of $G_p$ is $|E(G_p)|=\frac{(p^2-1)(2p^3+p^2-2p-2)}{2}.$ 
	\end{corollary}
	
	\begin{proof}
		There are $(p+1)(p-1)=p^2-1$ vertices with $\Ker(A)=\Image(A)$, and $p(p+1)(p-1)$ vertices with $\Ker(A)\neq \Image(A)$. From Corollary~\ref{cor:degrees}, summing degrees and dividing by $2$, we have
		\begin{align*}
			2|E(G_p)|
			&=(p^2-1)(2p^2-p-2)+p(p+1)(p-1)(2p^2-p-1)\\
			&=(p^2-1)(2p^3+p^2-2p-2).
		\end{align*}
	\end{proof}
	
	\medskip
	The block model also determines two extremal invariants exactly. These formulas use more than the degree sequence: they depend on the directed incidence hidden in the ordered pair $(\Ker(A),\Image(A))$.
	
	\begin{theorem}\label{thm:cliqueodd}
		For every odd prime $p$,
		$$
		\omega(G_p)=p+1.
		$$
	\end{theorem}
	
	\begin{proof}
		Fix $L\in\calL$. The class $\calC_{L,L}$ is a clique of order $p-1$. Choose lines $M,N\in\calL\setminus\{L\}$, one vertex $x\in\calC_{L,M}$, and one vertex $y\in\calC_{N,L}$. Every vertex of $\calC_{L,L}$ is adjacent to both $x$ and $y$, while $x$ and $y$ are adjacent because the image line of $y$ is $L$, the kernel line of $x$. Hence $G_p$ contains a clique of order $(p-1)+2=p+1$.
		
		For the reverse inequality, let $K$ be a clique. Suppose first that $K$ contains a vertex from some diagonal type $\calC_{L,L}$. Any other vertex of $K$ must then have type $(L,M)$ or $(N,L)$. Among the off-diagonal types $(L,M)$ with fixed first coordinate $L$, two distinct types are not adjacent; the same is true among the types $(N,L)$ with fixed second coordinate $L$. Moreover, each off-diagonal type class is independent. Thus $K$ contains at most one vertex of an outgoing type and at most one vertex of an incoming type, in addition to at most all $p-1$ vertices of $\calC_{L,L}$. Hence $|K|\le p+1$.
		
		Now suppose that $K$ contains no diagonal-type vertex. Pick a vertex of type $(L,M)$ with $L\neq M$. Every other type represented in $K$ must be of the form $(M,X)$ or $(Y,L)$. Since diagonal types are excluded, at most one type of the first form and at most one type of the second form can occur in a clique; within every such off-diagonal class at most one vertex can occur. Therefore $|K|\le3<p+1$. Combining the two cases gives $\omega(G_p)=p+1$.
	\end{proof}
	
	Next, we discuss the independence number of $G_{p}.$
	\begin{theorem}\label{thm:independenceodd}
		For every odd prime $p$,
		$$
		\alpha(G_p)=\frac{(p-1)(p+1)^2}{4}.
		$$
	\end{theorem}
	
	\begin{proof}
		Since $q=p+1$, where $p$ is odd, $q$ is even. So choose disjoint subsets $X,Y\subseteq\calL$ with $|X|=|Y|=q/2$, and let
		$$
		S=\bigcup_{L\in X,\,M\in Y}\calC_{L,M}.
		$$
		Because $X\cap Y=\varnothing$, every type in this union is off-diagonal. If $(L,M)$ and $(L',M')$ are two represented types, then $M\in Y$ and $L'\in X$, so $M\neq L'$; likewise $M'\neq L$. Theorem~\ref{thm:block} shows that no two vertices of $S$ are adjacent. Hence
		$$
		\alpha(G_p)\ge |S|=(p-1)|X||Y|=\frac{(p-1)(p+1)^2}{4}.
		$$
		
		Let $S'$ be an arbitrary independent set. First assume that it contains no vertex of diagonal type. Let $X'$ be the set of kernel lines and $Y'$ the set of image lines occurring among the represented types. We claim that $X'\cap Y'=\varnothing$. Indeed, if $L\in X'\cap Y'$, then $S'$ contains vertices from types $(L,M)$ and $(N,L)$. Since diagonal types are absent, these types are distinct, and Theorem~\ref{thm:block} makes the two vertices adjacent, a contradiction. Thus every represented type lies in $X'\times Y'$ with $X'\cap Y'=\varnothing$. Each such type contributes at most $p-1$ vertices, and therefore
		$$
		|S'|\le(p-1)|X'||Y'|\le(p-1)\frac{q^2}{4}.
		$$
		
		It remains to consider an independent set containing a diagonal-type vertex, say from $\calC_{L,L}$. It contains no second vertex from that class, and every other represented type avoids $L$ in both coordinates. Applying the preceding argument on the remaining $p$ lines gives
		$$
		|S'|\le 1+(p-1)\frac{p^2-1}{4}.
		$$
		For $p\ge3$ this is strictly smaller than $(p-1)(p+1)^2/4$, since the difference equals $(p^2-3)/2$. The lower construction is therefore optimal, proving the formula.
	\end{proof}
	
	\medskip
	Theorem~\ref{thm:block} converts the adjacency problem from multiplication in $M_2(\F_p)$ to an incidence problem on $\PP^1(\F_p)$. This reduction constitutes precisely the type of structural information that remains obscured by the scalar invariants in \cite{grau2017}. Specifically, it demonstrates that the occurrence of repeated rows in $A_p$ is not coincidental as they systematically arise inside each type class $\calC_{L,M}$. 
	
	\begin{example}\label{ex:p3types}
		Let $\calL=\{L_1,L_2,L_3,L_4\},$ $\calC_{rs}:=\calC_{L_r,L_s}$ for $1\le r,s\le 4$, and order the $16$ type classes lexicographically as
		$$
		\calC_{11},\calC_{12},\calC_{13},\calC_{14},
		\calC_{21},\calC_{22},\calC_{23},\calC_{24},
		\calC_{31},\calC_{32},\calC_{33},\calC_{34},
		\calC_{41},\calC_{42},\calC_{43},\calC_{44}.
		$$
		For $p=3$, each class has size $p-1=2$, so $|\calL|=4$, $G_3$ has $4^2=16$ type classes, each of size $2$, and hence
		$|V(G_3)|=32,$ see block diagram in Figure \ref{fig:G3-joined-union-manual}.
		By Corollary~\ref{cor:degrees}, the $8$ diagonal-type vertices have degree $13$, while the $24$ off-diagonal-type vertices have degree $14$. Corollary~\ref{cor:edges} gives $|E(G_3)|=220.$ 
		Thus, the matrix $A_3$ is a $32\times 32$ symmetric $(0,1)$ matrix built as a $16\times16$ block matrix with $2\times2$ blocks
		$$
		\renewcommand{\arraystretch}{1.05}
		\setlength{\arraycolsep}{2.8pt}
		A_3=
		\left(
		\begin{array}{cccc|cccc|cccc|cccc}
			K&J&J&J&J&O&O&O&J&O&O&O&J&O&O&O\\
			J&O&O&O&J&J&J&J&J&O&O&O&J&O&O&O\\
			J&O&O&O&J&O&O&O&J&J&J&J&J&O&O&O\\
			J&O&O&O&J&O&O&O&J&O&O&O&J&J&J&J\\
			\hline
			J&J&J&J&O&J&O&O&O&J&O&O&O&J&O&O\\
			O&J&O&O&J&K&J&J&O&J&O&O&O&J&O&O\\
			O&J&O&O&O&J&O&O&J&J&J&J&O&J&O&O\\
			O&J&O&O&O&J&O&O&O&J&O&O&J&J&J&J\\
			\hline
			J&J&J&J&O&O&J&O&O&O&J&O&O&O&J&O\\
			O&O&J&O&J&J&J&J&O&O&J&O&O&O&J&O\\
			O&O&J&O&O&O&J&O&J&J&K&J&O&O&J&O\\
			O&O&J&O&O&O&J&O&O&O&J&O&J&J&J&J\\
			\hline
			J&J&J&J&O&O&O&J&O&O&O&J&O&O&O&J\\
			O&O&O&J&J&J&J&J&O&O&O&J&O&O&O&J\\
			O&O&O&J&O&O&O&J&J&J&J&J&O&O&O&J\\
			O&O&O&J&O&O&O&J&O&O&O&J&J&J&J&K
		\end{array}
		\right),
		$$
		where each displayed entry is itself a $2\times2$ block. Equivalently, $A_3\sim H_3\otimes \J_2-D_3\otimes \I_2,$ where the $16\times16$ matrix $H_3$ is obtained from the above display by replacing every $K$ with $J$. From $H_3\otimes \J_2$, the four diagonal-type classes $\calC_{11},\calC_{22},\calC_{33},\calC_{44}$ constitute the block $K$.
	\end{example}

	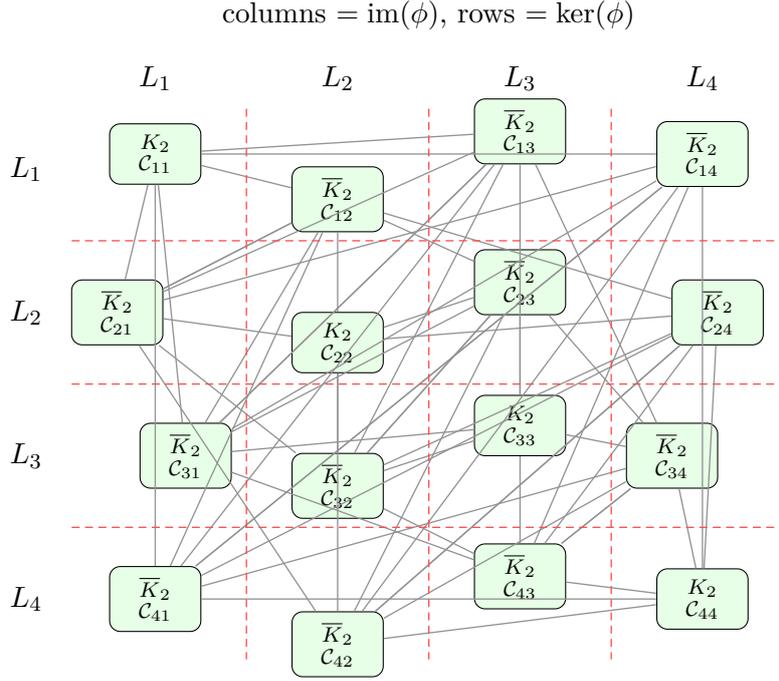
\begin{figure}[H]
		\centering
		\begin{tikzpicture}[
			scale=1,
			super/.style={draw, rounded corners, minimum width=12mm, minimum height=8mm,
				fill=green!10, font=\scriptsize, align=center},
			superedge/.style={gray!85, line width=0.5pt}
			]
			
			\node[super] (c11) at (0.0,  0.2) {$K_2$\\[-0.5mm]$\mathcal{C}_{11}$};
			\node[super] (c12) at (2.4,  -0.4) {$\overline{K}_2$\\[-0.5mm]$\mathcal{C}_{12}$};
			\node[super] (c13) at (4.8,  0.5) {$\overline{K}_2$\\[-0.5mm]$\mathcal{C}_{13}$};
			\node[super] (c14) at (7.2,  0.2) {$\overline{K}_2$\\[-0.5mm]$\mathcal{C}_{14}$};
			
			\node[super] (c21) at (-.5, -1.9) {$\overline{K}_2$\\[-0.5mm]$\mathcal{C}_{21}$};
			\node[super] (c22) at (2.4, -2.3) {$K_2$\\[-0.5mm]$\mathcal{C}_{22}$};
			\node[super] (c23) at (4.8, -1.5) {$\overline{K}_2$\\[-0.5mm]$\mathcal{C}_{23}$};
			\node[super] (c24) at (7.4, -1.9) {$\overline{K}_2$\\[-0.5mm]$\mathcal{C}_{24}$};
			
			\node[super] (c31) at (0.4, -3.8) {$\overline{K}_2$\\[-0.5mm]$\mathcal{C}_{31}$};
			\node[super] (c32) at (2.4, -4.2) {$\overline{K}_2$\\[-0.5mm]$\mathcal{C}_{32}$};
			\node[super] (c33) at (4.8, -3.4) {$K_2$\\[-0.5mm]$\mathcal{C}_{33}$};
			\node[super] (c34) at (6.8, -3.8) {$\overline{K}_2$\\[-0.5mm]$\mathcal{C}_{34}$};
			
			\node[super] (c41) at (0.0, -5.7) {$\overline{K}_2$\\[-0.5mm]$\mathcal{C}_{41}$};
			\node[super] (c42) at (2.4, -6.3) {$\overline{K}_2$\\[-0.5mm]$\mathcal{C}_{42}$};
			\node[super] (c43) at (4.8, -5.4) {$\overline{K}_2$\\[-0.5mm]$\mathcal{C}_{43}$};
			\node[super] (c44) at (7.2, -5.7) {$K_2$\\[-0.5mm]$\mathcal{C}_{44}$};
			
			\foreach \r in {1,...,4}{
				\foreach \s in {1,...,4}{
					\foreach \u in {1,...,4}{
						\foreach \v in {1,...,4}{
							\pgfmathtruncatemacro{\idxA}{4*\r+\s}
							\pgfmathtruncatemacro{\idxB}{4*\u+\v}
							\ifnum\idxA<\idxB
							\ifnum\s=\u \draw[superedge] (c\r\s) -- (c\u\v); \fi
							\ifnum\v=\r \draw[superedge] (c\r\s) -- (c\u\v); \fi
							\fi
						}
					}
				}
			}
			
			\begin{scope}[densely dashed, red!70, line width=0.5pt]
				\draw (-1.1, -0.95) -- (8.3, -0.95);
				\draw (-1.1, -2.85) -- (8.3, -2.85);
				\draw (-1.1, -4.75) -- (8.3, -4.75);
				\draw (1.2, 0.8) -- (1.2, -6.5);
				\draw (3.6, 0.8) -- (3.6, -6.5);
				\draw (6.0, 0.8) -- (6.0, -6.5);
			\end{scope}
			
			\foreach \i in {1,...,4}{
				\node[font=\small] at (-1.7, -1.9*\i+1.9) {$L_{\i}$};
				\node[font=\small] at (2.4*\i-2.4, 1.2) {$L_{\i}$};
			}
			\node[font=\small, anchor=south] at (3.6, 1.7) {columns $= \text{im}(\phi)$, rows $= \ker(\phi)$};
			
		\end{tikzpicture}
		\caption{A block representation of $G_3$.}
		\label{fig:G3-joined-union-manual}
	\end{figure}
	Figure \ref{fig:G3-joined-union-manual} presents an exact clustered drawing of the graph $G_3$, where each type class $\calC_{rs}=\calC_{L_r,L_s}$ consists of two vertices. If $r=s$, the class induces $K_2$, and if $r\neq s$, it induces $\overline{K}_2$. Two distinct classes $\calC_{rs}$ and $\calC_{uv}$ are joined by all four edges of a $K_{2,2}$
	exactly when $s=u$ or $v=r$.
	\section{Spectral consequences for odd prime modulus}\label{sec:spectral}
	
	We now exploit the block decomposition from Theorem~\ref{thm:block} to obtain new spectral information. In this sense,  the following theorem shows that all forced eigenvalues arising from row repetition can be isolated explicitly, and that the remaining spectrum lives on a much smaller reduced matrix of order $(p+1)^2$.
	
	\begin{theorem}\label{thm:charfactor}
		Let $p$ be an odd prime, and let $A_p=A(G_p)$. For each type class $\calC_{L,M}$, let
		$\one_{L,M}\in \mathbb{R}^{V(G_p)}$ be its indicator vector. Let $U=\operatorname{span}\{\one_{L,M}:\ (L,M)\in \calT\},$ be the class-constant subspace, and for each class $\calC_{L,M}$ let
		$$
		W_{L,M}=\left\{x\in \mathbb{R}^{V(G_p)}:\ \operatorname{supp}(x)\subseteq \calC_{L,M},\ \sum_{v\in \calC_{L,M}}x_v=0\right\}.
		$$
		Then the following hold.
		\begin{enumerate}[label=\textup{(\roman*)},noitemsep]
			\item $\mathbb{R}^{V(G_p)}=U\oplus\bigoplus_{(L,M)\in\calT}W_{L,M}$ is an $A_p$-invariant direct sum decomposition.
			\item On $U$, the matrix of $A_p$ relative to the basis $\{\one_{L,M}\}_{(L,M)\in\calT}$ is $B_p:=(p-1)H_p-D_p.$ 
			\item If $L\neq M$, then $A_p|_{W_{L,M}}=0$.
			\item If $L=M$, then $A_p|_{W_{L,L}}=-I$.
		\end{enumerate}
		Consequently, the characteristic polynomial of $A_p$ is
		$$
		\chi_{A_p}(\lambda)
		=
		\lambda^{\,p(p+1)(p-2)}(\lambda+1)^{(p+1)(p-2)}\chi_{B_p}(\lambda).
		$$
	\end{theorem}
	\begin{proof}
		For each type $(L,M)$, the coordinate space on $\calC_{L,M}$ decomposes as
		$$
		\mathbb R^{\calC_{L,M}}=\mathbb R\one_{L,M}\oplus W_{L,M},
		\qquad \dim W_{L,M}=p-2.
		$$
		Summing over all $(p+1)^2$ types gives the direct sum in part \textup{(i)} and $\dim U=(p+1)^2$.
		
		Let $x\in W_{L,M}$. Every off-diagonal block acting on $x$ is either $\O_{p-1}$ or $\J_{p-1}$, and $\J_{p-1}x=0$ because the coordinates of $x$ sum to zero. Hence $W_{L,M}$ is invariant. Its diagonal block is $\O_{p-1}$ when $L\neq M$, giving eigenvalue $0$, and $\J_{p-1}-\I_{p-1}$ when $L=M$, giving eigenvalue $-1$. This proves \textup{(iii)} and \textup{(iv)}.
		
		The block structure also maps class-constant vectors to class-constant vectors, so $U$ is invariant. On the basis $\{\one_{L,M}\}$, an off-diagonal incident type contributes $(p-1)$, while a diagonal type $(L,L)$ contributes $p-2=(p-1)-1$. Therefore the matrix of $A_p|_U$ is
		$$
		B_p=(p-1)H_p-D_p,
		$$
		which proves \textup{(ii)}.
		
		There are $p(p+1)$ off-diagonal types and $p+1$ diagonal types. Each corresponding $W_{L,M}$ has dimension $p-2$. Multiplying the characteristic polynomials of the invariant summands yields
		$$
		\chi_{A_p}(\lambda)
		=\lambda^{p(p+1)(p-2)}(\lambda+1)^{(p+1)(p-2)}\chi_{B_p}(\lambda).
		$$
	\end{proof}
	
	\medskip
	Theorem~\ref{thm:charfactor} has the following immediate corollaries.
	
	\begin{corollary}\label{cor:nullity}
		For every odd prime $p$, $\nullity(A_p)\ge p(p+1)(p-2),$ and $m_{A_p}(-1)\ge (p+1)(p-2),$ where $m_{A_p}(-1)$ denotes the multiplicity of the eigenvalue $-1$. 	Hence $\rank(A_p)\le (p+1)(2p-1).$
	\end{corollary}
	
	\begin{proof}
		The first two inequalities are immediate from Theorem~\ref{thm:charfactor}, since $|V(G_p)|=(p+1)^2(p-1)$, and we have
		$$
		\rank(A_p)=|V(G_p)|-\nullity(A_p)
		\le (p+1)^2(p-1)-p(p+1)(p-2)
		=(p+1)(2p-1).
		$$
	\end{proof}
	
	\medskip
	The preceding factorization can in fact be completed. The extra step is to use the simultaneous permutation symmetry of the two projective-line coordinates in the reduced type space.
	
	\begin{theorem}\label{thm:fullreducedspectrum}
		Let $p$ be an odd prime and Let
		$
		\Delta_p=4p^4-4p^2-8p+9,~
		\Theta_p=p^4-8p+8.
		$
		Define
		$$
		\lambda_{\pm}=p^2-p-\frac12\pm\frac12\sqrt{\Delta_p},
		\quad
		\text{and}
		\quad
		\mu_{\pm}=\frac{p(p-2)\pm\sqrt{\Theta_p}}{2}.
		$$
		Then the spectrum of the reduced matrix $B_p=(p-1)H_p-D_p$ is
		$$
		\left\{
		\lambda_+,\lambda_-,
		\mu_+^{[p]},\mu_-^{[p]},
		\bigl(-p(p-1)\bigr)^{[p]},
		(p-1)^{[p(p-1)/2]},
		\bigl(-(p-1)\bigr)^{[(p+1)(p-2)/2]}
		\right\},
		$$
		where the superscript $[m]$ denotes multiplicity $m$. Equivalently,
		$$
		\begin{aligned}
			\chi_{B_p}(\lambda)
			={}&(\lambda-\lambda_+)(\lambda-\lambda_-)
			(\lambda-\mu_+)^p(\lambda-\mu_-)^p
			(\lambda+p(p-1))^p\\
			&\times(\lambda-(p-1))^{p(p-1)/2}
			(\lambda+(p-1))^{(p+1)(p-2)/2}.
		\end{aligned}
		$$
		Together with Theorem~\ref{thm:charfactor}, this gives the complete adjacency spectrum of $G_p$.
	\end{theorem}
	
	\begin{proof}
		Write $q=p+1$ and $r=p-1$. Label the lines of $\calL$ by $1,\ldots,q$ and identify $\mathbb R^{\calT}$ with the space of real $q\times q$ matrices $X=(x_{ij})$. If
		$$
		R_j=\sum_{\ell=1}^q x_{j\ell},
		\qquad
		C_i=\sum_{k=1}^q x_{ki},
		$$
		then the definition of $H_p$ gives
		$$
		(H_pX)_{ij}=R_j+C_i-x_{ji}.
		$$
		The subtraction of $x_{ji}$ removes the unique type $(j,i)$ counted in both incidence sums. Since $D_p$ keeps only diagonal coordinates,
		$$
		(B_pX)_{ij}=r(R_j+C_i-x_{ji})-\delta_{ij}x_{ii}.
		ag{4.1}
		$$
		
		Let $\one\in\mathbb R^q$ be the all-ones vector. We first consider
		$$
		\mathscr A=\{X:X^{\top}=-X,\ X\one=0\}.
		$$
		For $X\in\mathscr A$, both row and column sums vanish and $x_{ji}=-x_{ij}$, so (4.1) gives $B_pX=rX$. The dimension is
		$$
		\dim\mathscr A=\frac{(q-1)(q-2)}{2}=\frac{p(p-1)}{2}.
		$$
		Next let
		$$
		\mathscr S=\{X:X^{\top}=X,\ \operatorname{diag}(X)=0,\ X\one=0\}.
		$$
		Then (4.1) gives $B_pX=-rX$. The space of symmetric zero-diagonal matrices has dimension $q(q-1)/2$, and the row-sum map has rank $q$ for $q\ge3$; hence
		$$
		\dim\mathscr S=\frac{q(q-3)}{2}=\frac{(p+1)(p-2)}{2}.
		$$
		
		The two-dimensional space of matrices having one constant value on the diagonal and another off the diagonal is invariant. Relative to these two values, (4.1) acts as
		$$
		Q_p=
		\begin{pmatrix}
			p-2&2p(p-1)\\
			2(p-1)&(2p-1)(p-1)
		\end{pmatrix},
		$$
		whose eigenvalues are $\lambda_+$ and $\lambda_-$.
		
		It remains to account for the standard-coordinate part. Choose $u\in\one^{\perp}$ and consider
		$$
		\mathscr V(u)=\operatorname{span}\{u\one^{\top},\ \one u^{\top},\ \operatorname{diag}(u)\}.
		$$
		Equation (4.1) shows that this space is invariant. Relative to the displayed ordered basis, the representing matrix is
		$$
		M_p=
		\begin{pmatrix}
			0&p(p-1)&p-1\\
			p(p-1)&0&p-1\\
			-1&-1&-p
		\end{pmatrix}.
		$$
		The vector $(1,-1,0)^{\top}$ is an eigenvector with eigenvalue $-p(p-1)$. On the invariant plane where the first two coordinates agree, $M_p$ reduces to
		$$
		\begin{pmatrix}
			p(p-1)&p-1\\
			-2&-p
		\end{pmatrix},
		$$
		whose characteristic polynomial is
		$$
		\lambda^2-p(p-2)\lambda-(p-1)(p^2-2).
		$$
		Its two roots are exactly $\mu_+$ and $\mu_-$. Taking an orthogonal basis $u_1,\ldots,u_p$ of $\one^{\perp}$ therefore supplies each of these three eigenvalues with multiplicity $p$.
		
		The spaces just described are mutually orthogonal for the Frobenius inner product after the choice of an orthogonal basis of $\one^{\perp}$. Their dimensions add to
		$$
		2+\frac{p(p-1)}{2}+\frac{(p+1)(p-2)}{2}+3p=(p+1)^2.
		$$
		Thus they exhaust the reduced type space. Collecting the eigenvalues and their multiplicities gives the stated spectrum and characteristic polynomial.
	\end{proof}
	
	\medskip
	The following consequence gives nullity, multiplicity of $-1$ and the rank of $A_{p}.$
	\begin{corollary}\label{cor:exactnullityrank}
		For every odd prime $p$,
		$$
		\nullity(A_p)=p(p+1)(p-2),~
		m_{A_p}(-1)=(p+1)(p-2),
		$$
		and
		$
		\rank(A_p)=(p+1)(2p-1).
		$
	\end{corollary}
	
	\begin{proof}
		By Theorem~\ref{thm:fullreducedspectrum}, none of the eigenvalues $\pm(p-1)$ or $-p(p-1)$ is $0$ or $-1$ when $p\ge3$. The product $\lambda_+\lambda_-=-(p-1)(2p^2+p-2)$ and the product $\mu_+\mu_-=-(p-1)(p^2-2)$ are nonzero, so $0$ is not an eigenvalue of $B_p$. Moreover, if
		$$
		f_p(\lambda)=\lambda^2-(2p^2-2p-1)\lambda-(p-1)(2p^2+p-2),
		$$
		and
		$$
		g_p(\lambda)=\lambda^2-p(p-2)\lambda-(p-1)(p^2-2),
		$$
		then
		$$
		f_p(-1)=-(p-1)(2p^2-p-2)\neq0,
		$$
		and
		$$
		g_p(-1)=-(p-1)(p^2-p-1)\neq0.
		$$
		Hence $B_p$ contributes neither additional zero eigenvalues nor additional copies of $-1$. The first two equalities now follow from Theorem~\ref{thm:charfactor}. Finally,
		$$
		\rank(A_p)=|V(G_p)|-\nullity(A_p)
		=(p+1)^2(p-1)-p(p+1)(p-2)
		=(p+1)(2p-1).
		$$
	\end{proof}
	
	\medskip
	We next identify a small quotient matrix that controls the spectral radius. Let $\calD=\bigcup_{L\in\calL}\calC_{L,L},$ and $	\calO=\bigcup_{\substack{L,M\in\calL\\L\neq M}}\calC_{L,M}.$ Thus $\calD$ is the set of diagonal-type vertices, and $\calO$ is the set of off-diagonal-type vertices.
	
	\begin{proposition}\label{prop:equitable}
		The partition $V(G_p)=\calD\sqcup \calO$ is equitable, and its quotient matrix is
		$$
		Q_p=
		\begin{pmatrix}
			p-2 & 2p(p-1)\\[1mm]
			2(p-1) & (2p-1)(p-1)
		\end{pmatrix}.
		$$
	\end{proposition}
	
	\begin{proof}
		For $A\in \calD$, say $A\in\calC_{L,L}$, and from Corollary~\ref{cor:degrees}, the number of neighbors in $\calD$ is $p-2$, all inside the same class $\calC_{L,L}$. So, the remaining neighbors lie in $\calO$, and their number is $(2p^2-p-2)-(p-2)=2p(p-1).$ Thus, every vertex in $\calD$ has row $(p-2,\,2p(p-1))$ in the quotient. Now take $A\in \calO$, say $A\in\calC_{L,M}$ with $L\neq M$. The only diagonal-type neighboring classes are $\calC_{L,L}$ and $\calC_{M,M}$, each of size $p-1$, so $A$ has $2(p-1)$ neighbors in $\calD$. Since $\deg(A)=(2p+1)(p-1)$, the number of neighbors in $\calO$ is $(2p+1)(p-1)-2(p-1)=(2p-1)(p-1).$ 	Therefore, the partition is equitable with the stated quotient matrix.
	\end{proof}
	
	\medskip
	The following theorem presents a closed formula for the spectral radius of $A_p$. 	
	\begin{theorem}\label{thm:specradius}
		Let $p$ be an odd prime. Then
		$$
		\rho(A_p)
		=
		p^2-p-\frac12+\frac12\sqrt{4p^4-4p^2-8p+9}.
		$$
	\end{theorem}
	
	\begin{proof}
		By Proposition~\ref{prop:equitable}, the partition $V(G_p)=\calD\sqcup\calO$ is equitable with quotient matrix
		$$
		Q_p=
		\begin{pmatrix}
			p-2 & 2p(p-1)\\[1mm]
			2(p-1) & (2p-1)(p-1)
		\end{pmatrix}.
		$$
		Its characteristic polynomial is
		$$
		\lambda^2-(2p^2-2p-1)\lambda-(p-1)(2p^2+p-2),
		$$
		and therefore its larger eigenvalue is
		$$
		\lambda_+=p^2-p-\frac12+\frac12\sqrt{4p^4-4p^2-8p+9}.
		$$
		The matrix $Q_p$ is nonnegative and irreducible, so Perron--Frobenius gives a positive eigenvector $(x,y)^{\top}$ for $\lambda_+$. Lift it to a vector $z\in\mathbb R^{V(G_p)}$ by assigning the value $x$ to every vertex of $\calD$ and the value $y$ to every vertex of $\calO$. Equitability gives $A_pz=\lambda_+z$. Since $z$ is strictly positive, the Perron--Frobenius theorem for the nonnegative matrix $A_p$ implies that the corresponding eigenvalue is the spectral radius. Hence $\rho(A_p)=\lambda_+$.
	\end{proof}

	\medskip
	The following example illustrates these facts.
	\begin{example}\label{ex:p3spectral}
		For $p=3$, the two cells of the equitable partition are $\calD=\bigcup_{L\in\calL}\calC_{L,L},$ and $\calO=\bigcup_{\substack{L,M\in\calL\\L\neq M}}\calC_{L,M}.$  Also $|\calD|=(p+1)(p-1)=4\cdot 2=8,$ and $|\calO|=p(p+1)(p-1)=3\cdot 4\cdot 2=24.$ By Proposition~\ref{prop:equitable}, the quotient matrix of the partition
		$V(G_3)=\calD\sqcup \calO$ is
		$$
		Q_3=
		\begin{pmatrix}
			p-2 & 2p(p-1)\\[1mm]
			2(p-1) & (2p-1)(p-1)
		\end{pmatrix}
		=
		\begin{pmatrix}
			1&12\\
			4&10
		\end{pmatrix}.
		$$
		Its characteristic polynomial  is $\chi_{Q_3}(\lambda) 	=	\lambda^2-11\lambda-38.$  By Theorem~\ref{thm:specradius}, the spectral radius is  $\rho(A_3)=\tfrac{11+\sqrt{273}}{2}\approx 13.7614.$ 		
		Next, Corollary~\ref{cor:nullity} gives  $\nullity(A_3)\ge 3\cdot 4\cdot 1=12,$ and $m_{A_3}(-1)\ge 4\cdot 1=4.$ These bounds are sharp by Corollary~\ref{cor:exactnullityrank}: $\nullity(A_3)=12$, $m_{A_3}(-1)=4$, and $\rank(A_3)=20$. Thus the block-repetition subspaces account for the entire kernel and the entire $-1$-eigenspace in this case.
	\end{example}
	
	\section{The two-adic family and the exact base case $n=2$}\label{sec:twopower}
	
	The analysis of odd primes was centered on the matrix-ring model $M_2(\F_p)$. In the two-adic context, no general model exists \cite{grau2017}, requiring an alternative approach to the adjacency matrix.
	
	\medskip
	The following theorem determines the exact base graph at $n=2$.	
	\begin{theorem}\label{thm:n2}
		In $\LL_2$, let $u=1+i$ and $v=1+j$. Then $\LL_2\cong \F_2[u,v]/(u^2,v^2)$, the zero-divisor graph satisfies $G_2\cong F_3$ (Figure \ref{fig:G2}), and the adjacency matrix $A_2$ has characteristic polynomial
		$$
		\chi_{A_2}(\lambda)=(\lambda-3)(\lambda+2)(\lambda-1)^2(\lambda+1)^3.
		$$
	\end{theorem}
	
	\begin{proof}
		Work in characteristic $2$. The quaternion relations become $i^2=j^2=k^2=1$ and $ij=k=ji$, because the sign change in $ji=-k$ disappears modulo $2$. Put $u=1+i$ and $v=1+j$. Then
		$
		u^2=(1+i)^2=1+i^2=0,~
		v^2=(1+j)^2=1+j^2=0,
		$
		and
		$
		uv=(1+i)(1+j)=1+i+j+k.
		$
		The four elements $1,u,v,uv$ are linearly independent over $\F_2$ because the change of coordinates from the basis $1,i,j,k$ is invertible. Hence every element of $\LL_2$ has a unique expression $a+bu+cv+duv$, with $a,b,c,d\in\F_2$. The relations $u^2=v^2=0$ therefore give an isomorphism
		$$
		\LL_2\cong\F_2[u,v]/(u^2,v^2).
		$$
		
		Let $\mathfrak m=(u,v)$. The quotient $\LL_2/\mathfrak m$ is $\F_2$, so $\mathfrak m$ is maximal. If $a=1$ in an element $a+bu+cv+duv$, then its nilpotent part lies in $\mathfrak m$ and $1+(bu+cv+duv)$ is a unit; thus $\mathfrak m$ is the unique maximal ideal. Consequently the nonzero zero divisors are exactly
		$$
		uv,\ u,\ u+uv,\ v,\ v+uv,\ u+v,\ u+v+uv.
		$$
		
		For $x=bu+cv+duv$ and $y=b'u+c'v+d'uv$ in $\mathfrak m$, the defining relations give
		$$
		xy=(bc'+cb')uv.
		ag{5.1}
		$$
		Thus $uv$ annihilates all of $\mathfrak m$. Among the six remaining vertices, equation (5.1) shows that two distinct elements are adjacent exactly when their $(u,v)$-coefficient pairs are equal. The three nonzero pairs $(1,0)$, $(0,1)$, and $(1,1)$ therefore produce precisely the three edges
		$$
		u\sim u+uv,
		\qquad
		v\sim v+uv,
		\qquad
		u+v\sim u+v+uv.
		$$
		Together with the six edges incident to $uv$, this is the friendship graph $F_3$. In the ordering stated in the theorem, its adjacency matrix is exactly the displayed matrix $A_2$.
		
		Finally, $F_3$ has an equitable partition consisting of its central vertex and the six outer vertices, with quotient matrix
		$$
		\begin{pmatrix}0&6\\1&1\end{pmatrix},
		$$
		whose eigenvalues are $3$ and $-2$. On the outer vertices, differences inside each of the three matched pairs give the eigenvalue $-1$ with multiplicity $3$, while the two-dimensional subspace of vectors constant on each pair and having total pair-sum zero gives the eigenvalue $1$ with multiplicity $2$. Hence
		$$
		\chi_{A_2}(\lambda)=(\lambda-3)(\lambda+2)(\lambda-1)^2(\lambda+1)^3.
		$$
	\end{proof}
	
	\begin{figure}[H]
		\centering
		\begin{tikzpicture}[scale=1.1, every node/.style={circle,draw,inner sep=1.8pt,fill=white}]
			\node (c) at (0,0) {$uv$};
			\node (a1) at (2,0.9) {$u$};
			\node (a2) at (2,-0.9) {$u+uv$};
			\node (b1) at (-2,0.9) {$v$};
			\node (b2) at (-2,-0.9) {$v+uv$};
			\node (d1) at (2,2.2) {$u+v$};
			\node (d2) at (0,-2.2) {$u+v+uv$};
			
			\draw (c)--(a1)--(a2)--(c);
			\draw (c)--(b1)--(b2)--(c);
			\draw (c)--(d1)--(d2)--(c);
		\end{tikzpicture}
		\caption{The graph $G_2$.}
		\label{fig:G2}
	\end{figure}
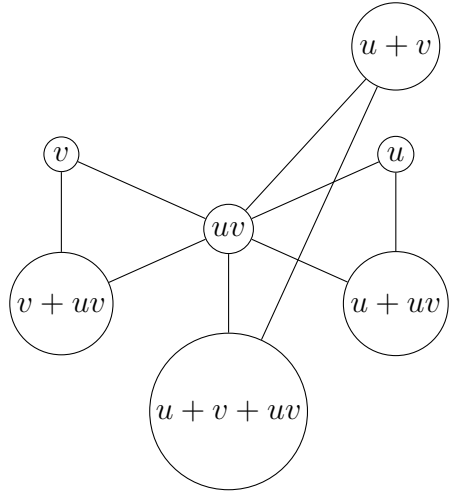
	
	\medskip
	The preceding exact calculation suggests using ideals to locate large principal blocks in the higher two-adic graphs. We first isolate the general mechanism.
	
	\begin{lemma}\label{lem:squarezero}
		Let $R$ be a finite ring and let $I$ be a nonzero ideal satisfying $I^2=0$. Then $I\setminus\{0\}$ induces a complete subgraph of the undirected zero-divisor graph of $R$.
	\end{lemma}
	
	\begin{proof}
		If $0\neq x\in I$, then $x^2=0$, so $x$ is a zero divisor. If $x,y\in I\setminus\{0\}$ are distinct, then $xy=0$ because $I^2=0$. Hence every two distinct nonzero elements of $I$ are adjacent.
	\end{proof}
	
	\medskip
	In $\LL_{2^t}$ the ideal $2^{\lceil t/2\rceil}\LL_{2^t}$ is square-zero and has an explicit order; this gives the quantitative statement below.
	\begin{theorem}\label{thm:clique2power}
		Let $t\ge 2$ and put $s=\left\lceil \tfrac{t}{2}\right\rceil$. Then the set $\bigl(2^s\LL_{2^t}\bigr)\setminus\{0\}$ induces a complete subgraph of $G_{2^t}$. Consequently,
		$$
		\omega(G_{2^t})\ge c_t:=|2^s\LL_{2^t}|-1
		=2^{4(t-s)}-1
		=2^{4\lfloor t/2\rfloor}-1,
		$$
		and, after a suitable permutation of vertices,
		$$
		A_{2^t}\sim
		\begin{pmatrix}
			\J_{c_t}-\I_{c_t} & *\\
			* & *
		\end{pmatrix}.
		$$
	\end{theorem}
	
	\begin{proof}
		As an additive group, $\LL_{2^t}\cong(\mathbb Z_{2^t})^4$. Hence
		$$
		|2^s\LL_{2^t}|=|2^s\mathbb Z_{2^t}|^4=2^{4(t-s)}.
		$$
		Moreover,
		$$
		\bigl(2^s\LL_{2^t}\bigr)^2\subseteq 2^{2s}\LL_{2^t}=0,
		$$
		because $2s\ge t$. Thus $2^s\LL_{2^t}$ is a square-zero ideal, and Lemma~\ref{lem:squarezero} shows that its nonzero elements induce a clique. Its order is $2^{4(t-s)}-1$, and the identity $t-\lceil t/2\rceil=\lfloor t/2\rfloor$ gives the displayed value of $c_t$. Ordering these clique vertices first makes the corresponding principal block equal to the adjacency matrix of $K_{c_t}$, namely $\J_{c_t}-\I_{c_t}$.
	\end{proof}
	
	\medskip
	The square-zero-ideal observation is general, while the explicit choice $2^{\lceil t/2\rceil}\LL_{2^t}$ and its cardinality are specific to the present two-adic family. This produces a concrete complete principal block $\J_{c_t}-\I_{c_t}$ and, in turn, explicit spectral and combinatorial bounds. 
	
	\medskip
	The following is an immediate consequence.
	\begin{corollary}\label{cor:2powerbounds}
		For $t\ge 2$, $\rho(A_{2^t})\ge c_t-1=2^{4\lfloor t/2\rfloor}-2,$ 	and $	|E(G_{2^t})|\ge \binom{c_t}{2}.$ 
	\end{corollary}
	
	\begin{proof}
		The adjacency matrix of the clique $K_{c_t}$ is $\J_{c_t}-\I_{c_t}$, whose spectral radius is $c_t-1$. Since this matrix is a principal submatrix of $A_{2^t}$, so Perron monotonicity for nonnegative symmetric matrices implies that  $\rho(A_{2^t})\ge c_t-1.$ The edge bound is immediate,  as $G_{2^t}$ contains a $K_{c_t}$ subgraph.
	\end{proof}
	
	\medskip
	The powers of $2$ also contain nested annihilating layers. They give complete bipartite subgraphs that can be substantially larger, in spectral terms, than the square-zero clique.
	
	\begin{theorem}\label{thm:annihilatorlayers}
		Let $t\ge2$, let $0\le a<b<t$ satisfy $a+b\ge t$, and let
		$
		I_r=2^r\LL_{2^t}.
		$
		Then the two disjoint sets
		$
		X=I_a\setminus I_b,
		$ and $
		Y=I_b\setminus\{0\}
		$
		are completely joined in $G_{2^t}$. Consequently $G_{2^t}$ contains a subgraph
		$
		K_{m_{a,b},n_b},
		$
		where
		$
		m_{a,b}=2^{4(t-a)}-2^{4(t-b)},
		$ and $
		n_b=2^{4(t-b)}-1.
		$
		In particular,
		$
		\rho(A_{2^t})\ge\sqrt{m_{a,b}n_b},
		$ and $
		|E(G_{2^t})|\ge m_{a,b}n_b.
		$
	\end{theorem}
	
	\begin{proof}
		The ideals are nested because $a<b$, so $I_b\subset I_a$ and the sets $X$ and $Y$ are disjoint. Since the scalar powers of $2$ are central,
		$$
		I_aI_b\subseteq2^{a+b}\LL_{2^t}=0,
		\qquad
		I_bI_a\subseteq2^{a+b}\LL_{2^t}=0,
		$$
		because $a+b\ge t$. Thus every $x\in X$ and $y\in Y$ satisfy both $xy=0$ and $yx=0$, so every cross pair is an edge of $G_{2^t}$.
		
		As additive groups, $I_r\cong(2^r\mathbb Z_{2^t})^4$, and hence $|I_r|=2^{4(t-r)}$. Therefore
		$$
		|X|=|I_a|-|I_b|=m_{a,b},
		\qquad
		|Y|=|I_b|-1=n_b.
		$$
		The cross edges give the stated complete bipartite subgraph and the edge bound. The adjacency matrix on $X\cup Y$ entrywise dominates the adjacency matrix of $K_{m_{a,b},n_b}$, whose spectral radius is $\sqrt{m_{a,b}n_b}$. Perron monotonicity on this principal vertex set, followed by principal-submatrix monotonicity, yields the spectral-radius bound.
	\end{proof}
	\medskip
	
	The following consequence follows from above result.
	\begin{corollary}\label{cor:oddtexplicit}
		If $t=2m+1\ge3$, then
		$$
		\rho(A_{2^t})\ge
		\sqrt{15\cdot2^{4m}\bigl(2^{4m}-1\bigr)}.
		$$
		This is strictly stronger than the clique bound $\rho(A_{2^t})\ge2^{4m}-2$ from Corollary~\ref{cor:2powerbounds}.
	\end{corollary}
	
	\begin{proof}
		Apply Theorem~\ref{thm:annihilatorlayers} with $a=m$ and $b=m+1$. Then
		$$
		m_{m,m+1}=2^{4(m+1)}-2^{4m}=15\cdot2^{4m},
		$$
		and $n_{m+1}=2^{4m}-1$. For $m\ge1$, the square of the new bound exceeds $(2^{4m}-2)^2$, so the inequality is strict.
	\end{proof}
	
	\medskip
	
	The next result discuss the inertia.
	\begin{theorem}\label{thm:twoadic-inertia}
		Let $t\ge2$, let $N_t=|V(G_{2^t})|=2^{4t-1}-1$, and let $n_+(A_{2^t})$, $n_-(A_{2^t})$, and $n_0(A_{2^t})$ denote the positive, negative, and zero inertia indices. With $c_t=2^{4\lfloor t/2\rfloor}-1$ as in Theorem~\ref{thm:clique2power},
		$$
		n_+(A_{2^t})\ge1,
		\qquad
		n_-(A_{2^t})\ge c_t-1.
		$$
		Consequently,
		$$
		\rank(A_{2^t})\ge c_t,
		\qquad
		n_0(A_{2^t})\le N_t-c_t.
		$$
	\end{theorem}
	
	\begin{proof}
		By Theorem~\ref{thm:clique2power}, $A_{2^t}$ has a principal submatrix equal to $\J_{c_t}-\I_{c_t}$. Its spectrum is
		$$
		\{c_t-1,(-1)^{[c_t-1]}\},
		$$
		so its inertia is $(1,c_t-1,0)$. The inertia of a real symmetric principal submatrix is dominated componentwise by the inertia of the full matrix; equivalently, this follows from Cauchy interlacing. Hence $A_{2^t}$ has at least one positive and at least $c_t-1$ negative eigenvalues. The rank and nullity bounds follow immediately from
		$$
		\rank(A_{2^t})=n_+(A_{2^t})+n_-(A_{2^t})
		$$
		and $n_0(A_{2^t})=N_t-\rank(A_{2^t})$.
	\end{proof}
	
	\section{Algorithmic construction and numerical illustrations}\label{sec:algorithm}
	
	The block model of Section~\ref{sec:oddprime} leads to a natural algorithm for constructing $A_p$ when $p$ is an odd prime.
	Rather than verifying adjacency through the multiplication of all singular matrices in $M_2(\F_p)$, one might construct the matrix based on projective-line incidence. This replaces vertex-pair ring multiplications by incidence tests on the much smaller type set. The compressed block description is inexpensive to obtain, although expanding it into a dense matrix still has the unavoidable cost of writing all matrix entries.
	
	\begin{algorithm}[H]
		\caption{Structured construction of $A_p$ for an odd prime $p$}
		\label{alg:Ap}
		\begin{algorithmic}[1]
			\Require Odd prime $p$
			\Ensure Adjacency matrix $A_p$ of $G_p$
			\State Enumerate the $p+1$ lines of $\PP^1(\F_p)$ as $L_0,\dots,L_p$
			\State Form the set of types $\calT=\{(L_r,L_s):0\le r,s\le p\}$
			\State For each type $(L_r,L_s)$, create a class of $p-1$ vertex labels
			\State Initialize a zero matrix of order $(p+1)^2(p-1)$
			\For{each ordered pair of types $\alpha=(L_r,L_s)$, $\beta=(L_u,L_v)$}
			\If{$\alpha=\beta$ and $L_r=L_s$}
			\State place the block $\J_{p-1}-\I_{p-1}$ in position $(\alpha,\beta)$
			\ElsIf{$\alpha\neq\beta$ and $(L_s=L_u \text{ or } L_v=L_r)$}
			\State place the block $\J_{p-1}$ in position $(\alpha,\beta)$
			\Else
			\State place the block $\O_{p-1}$ in position $(\alpha,\beta)$
			\EndIf
			\EndFor
			\State \Return $A_p$
		\end{algorithmic}
	\end{algorithm}
	
	\begin{proposition}\label{prop:algcorrect}
		Algorithm~\ref{alg:Ap} returns the adjacency matrix of $G_p$. Its type-incidence stage uses $O(p^4)$ decisions. The compressed block description therefore has $O(p^4)$ decision complexity, whereas a naive vertex-pair method uses $O(|V(G_p)|^2)=O(p^6)$ ring-multiplication tests. If the full dense matrix is explicitly materialized, however, its $\Theta(p^6)$ entries force $\Theta(p^6)$ output and storage cost.
	\end{proposition}
	
	\begin{proof}
		Correctness is exactly Theorem~\ref{thm:block}. The type set has $(p+1)^2$ elements, so examining all ordered pairs of types requires $O((p+1)^4)=O(p^4)$ incidence decisions. On the other hand,
		$$
		|V(G_p)|=(p+1)^2(p-1)=\Theta(p^3),
		$$
		and a naive method that tests all vertex pairs by ring multiplication therefore performs $O(p^6)$ algebraic tests. The distinction is important: the $O(p^4)$ count concerns the compressed decision stage, not the cost of writing a dense adjacency matrix. Since $A_p$ has $\Theta(|V(G_p)|^2)=\Theta(p^6)$ entries, dense expansion and storage are necessarily $\Theta(p^6)$.
	\end{proof}
	
	We now collect numerical consequences of these facts in Tables \ref{tab:oddprimes}, \ref{tab:twopowers} and \ref{tab:complexity}.
	\begin{table}[H]
		\centering
		\caption{Numerical values of the new formulas for odd primes. Here $d_{\mathrm{diag}}$ and $d_{\mathrm{off}}$ are the two degree values from Corollary~\ref{cor:degrees}.}
		\label{tab:oddprimes}
		\begin{tabular}{ccccccc}
			\toprule
			$p$ & $|V(G_p)|$ & $d_{\mathrm{diag}}$ & $d_{\mathrm{off}}$ & $|E(G_p)|$ & $\nullity(A_p)\ge$ & $\rho(A_p)$ \\
			\midrule
			$3$ & $32$ & $13$ & $14$ & $220$ & $12$ & $\frac{11+\sqrt{273}}{2}\approx 13.7614$ \\
			$5$ & $144$ & $43$ & $44$ & $3156$ & $90$ & $\frac{39+\sqrt{2369}}{2}\approx 43.8362$ \\
			$7$ & $384$ & $89$ & $90$ & $17256$ & $280$ & $\frac{83+\sqrt{9361}}{2}\approx 89.8761$ \\
			\bottomrule
		\end{tabular}
	\end{table}
	
	\begin{table}[H]
		\centering
		\caption{Two-adic examples. The clique lower bound is the value $c_t$ from Theorem~\ref{thm:clique2power}; for $t=1$ the exact graph is given in Theorem~\ref{thm:n2}.}
		\label{tab:twopowers}
		\begin{tabular}{ccccc}
			\toprule
			$t$ & $n=2^t$ & $|V(G_{2^t})|$ & clique size available & spectral radius bound \\
			\midrule
			$1$ & $2$ & $7$ & exact graph $F_3$ & $\rho(A_2)=3$ \\
			$2$ & $4$ & $127$ & $15$ & $\rho(A_4)\ge 14$ \\
			$3$ & $8$ & $2047$ & $15$ & $\rho(A_8)\ge 14$ \\
			$4$ & $16$ & $32767$ & $255$ & $\rho(A_{16})\ge 254$ \\
			\bottomrule
		\end{tabular}
	\end{table}
	
	\begin{table}[H]
		\centering
		\caption{Comparison of adjacency-decision counts for naive vertex testing and the compressed type-based construction in Algorithm~\ref{alg:Ap}. Dense materialization of $A_p$ still requires $\Theta(p^6)$ output/storage.}
		\label{tab:complexity}
		\begin{tabular}{cccc}
			\toprule
			$p$ & $|V(G_p)|$ & naive unordered pair tests $\binom{|V|}{2}$ & type-incidence tests $(p+1)^4$ \\
			\midrule
			$3$ & $32$ & $496$ & $256$ \\
			$5$ & $144$ & $10296$ & $1296$ \\
			$7$ & $384$ & $73536$ & $4096$ \\
			\bottomrule
		\end{tabular}
	\end{table}
	Tables~\ref{tab:oddprimes}--\ref{tab:complexity} illustrate the principal numerical implications of the theory. Table~\ref{tab:oddprimes} presents the explicit formulas for odd primes, Table~\ref{tab:twopowers} illustrates the clique-based lower bounds within the two-adic family, and Table~\ref{tab:complexity} contrasts the structured construction of $A_p$ with naive adjacency testing. Together, the tables show the regularity created by the type partition and the reduction in adjacency-decision work obtained from the compressed block model. They do not change the $\Theta(p^6)$ size of the full dense output.

	\begin{example}
		For $p=5$, Theorem~\ref{thm:block} yields a $144\times 144$ adjacency matrix built from $36$ classes of size $4$.  From Theorem~\ref{thm:charfactor}, the reduced matrix $B_5$ has order $36$, while the quotient matrix $Q_5$ from Proposition~\ref{prop:equitable} already determines the spectral radius
		$$
		Q_5=
		\begin{pmatrix}
			3 & 40\\
			8 & 36
		\end{pmatrix},
		\qquad
		\rho(A_5)=\frac{39+\sqrt{2369}}{2}\approx 43.8362.
		$$
		At the same time, Corollary~\ref{cor:nullity} guarantees $\nullity(A_5)\ge 90$, so more than $62\%$ of the ambient space is forced into the kernel by the type repetition alone.
	\end{example}
	
	\begin{remark}
		Tables~\ref{tab:oddprimes}--\ref{tab:complexity} demonstrate two complementary aspects of the theory. The odd-prime family demonstrates a well-organized repetition: although the dense matrix is large, the reduced spectral problem is substantially smaller. Secondly, the two-adic family demonstrates large principal blocks, even in the absence of a perfect full reduction.
	\end{remark}
	
	%
	
	\section{Energy of the adjacency matrices}\label{sec:energy}
	
	The preceding adjacency-matrix decomposition also gives information about the \emph{energy} of the graph. Energy serves as a fundamental spectral invariant in algebraic graph theory and mathematical chemistry, as it quantifies the overall magnitude of the adjacency spectrum, represents the trace norm of the adjacency matrix, and frequently exhibits sensitivity to repeated rows, equitable quotients, and low-rank reductions \cite{gutman2001energy,li2012energy,nikiforov2007energy,nikiforov2016energy,bilalqm}.
	Sections~\ref{sec:spectral} and \ref{sec:twopower} provide two complementary inputs for energy estimates: a reduced spectral model in the odd-prime case and large complete principal blocks in the two-adic case.
	
	Energy complements the vertex, edge, spectral-radius, and nullity data by measuring the total absolute mass of the adjacency spectrum. In the present setting it also shows the contrast between the odd-prime family, where the spectrum has a substantial reduced component, and the two-adic family, where our general information comes from large forced subgraphs.
	
	Let $G$ be a finite simple graph with adjacency eigenvalues $\lambda_1,\lambda_2,\dots,\lambda_{|V(G)|}.$ The \emph{energy} of $G$ is $\mathcal E(G):=\sum_{r=1}^{|V(G)|} |\lambda_r|.$ Since $A(G)$ is real symmetric, this is exactly the trace norm: $\mathcal E(G)=\|A(G)\|_*$.  
	
	\medskip
	We record two elementary inequalities that will be used repeatedly.
	\begin{proposition}\label{prop:energybasic}
		Let $M$ be a nonzero real symmetric matrix with eigenvalues $\mu_1,\ldots,\mu_m$ and spectral radius $\rho(M)=\max_i|\mu_i|$. Then
		$$
		\|M\|_*\ge \frac{\tr(M^2)}{\rho(M)}.
		$$
		If $M=A(G)$ is the adjacency matrix of a finite simple graph $G$, then in addition
		$$
		\mathcal E(G)\ge 2\rho(A(G)),
		\qquad
		\mathcal E(G)\ge \frac{2|E(G)|}{\rho(A(G))}.
		$$
	\end{proposition}
	
	\begin{proof}
		Since $M$ is real symmetric,
		$$
		\tr(M^2)=\sum_{i=1}^m\mu_i^2
		\le \rho(M)\sum_{i=1}^m|\mu_i|
		=\rho(M)\|M\|_*,
		$$
		which proves the first inequality. If $M=A(G)$, then $\tr A(G)=0$. Hence the sum of the positive eigenvalues equals the absolute value of the sum of the negative eigenvalues, and therefore $\mathcal E(G)$ is twice the sum of the positive eigenvalues. Since $\rho(A(G))$ is a positive eigenvalue whenever $G$ has an edge, $\mathcal E(G)\ge2\rho(A(G))$. Finally, $\tr(A(G)^2)=2|E(G)|$, so the first inequality gives the second graph bound. The edgeless case is immediate.
	\end{proof}

	\medskip
	The next theorem gives an exact energy reduction for the odd-prime graphs. It isolates the contribution of the forced $-1$-eigenspaces and leaves the remaining energy in the reduced matrix of order $(p+1)^2$.	
	\begin{theorem}\label{thm:energyoddprime}
		Let $p$ be an odd prime, and let $B_p=(p-1)H_p-D_p$ be the reduced matrix from Theorem~\ref{thm:charfactor}. Then
		$\mathcal E(G_p)=(p+1)(p-2)+\mathcal E(B_p).$
	\end{theorem}
	
	\begin{proof}
		By Theorem~\ref{thm:charfactor}, we have
		$$
		\chi_{A_p}(\lambda)
		=
		\lambda^{\,p(p+1)(p-2)}(\lambda+1)^{(p+1)(p-2)}\chi_{B_p}(\lambda).
		$$
		Hence the eigenvalues of $A_p$ consist of the eigenvalues of $B_p$, together with $p(p+1)(p-2)$ zeros and $(p+1)(p-2)$ copies of $-1$. Taking absolute values and summing gives
		$$
		\mathcal E(G_p)=0\cdot p(p+1)(p-2)+1\cdot (p+1)(p-2)+\mathcal E(B_p).
		$$
	\end{proof}
	
	\medskip
	The complete reduced spectrum from Theorem~\ref{thm:fullreducedspectrum} turns the preceding reduction into a closed formula.
	
	\begin{theorem}\label{thm:exactenergyodd}
		For every odd prime $p$,
		$$
		\mathcal E(G_p)
		=2p^3-2p^2-p-1
		+\sqrt{4p^4-4p^2-8p+9}
		+p\sqrt{p^4-8p+8}.
		$$
		In particular, $\mathcal E(G_p)=3p^3+O(p^2)$ as $p\to\infty$.
	\end{theorem}
	
	\begin{proof}
		The two quotient eigenvalues $\lambda_+$ and $\lambda_-$ have opposite signs because their product is $-(p-1)(2p^2+p-2)$. Therefore
		$$
		|\lambda_+|+|\lambda_-|=\sqrt{4p^4-4p^2-8p+9}.
		$$
		Likewise, $\mu_+\mu_-=-(p-1)(p^2-2)<0$, so
		$$
		|\mu_+|+|\mu_-|=\sqrt{p^4-8p+8}.
		$$
		Using the multiplicities in Theorem~\ref{thm:fullreducedspectrum}, we obtain
		$$
		\begin{aligned}
			\mathcal E(B_p)
			={}&\sqrt{4p^4-4p^2-8p+9}
			+p\sqrt{p^4-8p+8}\\
			&+p^2(p-1)
			+\frac{p(p-1)^2}{2}
			+\frac{(p+1)(p-2)(p-1)}{2}\\
			={}&\sqrt{4p^4-4p^2-8p+9}
			+p\sqrt{p^4-8p+8}
			+(p-1)(2p^2-p-1).
		\end{aligned}
		$$
		Now apply Theorem~\ref{thm:energyoddprime} and use $(p+1)(p-2)=p^2-p-2$. This gives the claimed formula. The asymptotic statement follows directly from the two square-root terms.
	\end{proof}
	
	\medskip
	Next, we have the lower bound for 
	\begin{theorem}\label{thm:hyperodd}
		For every odd prime $p$, the graph $G_p$ is hyperenergetic; that is,
		$$
		\mathcal E(G_p)>2\bigl(|V(G_p)|-1\bigr)=\mathcal E\bigl(K_{|V(G_p)|}\bigr).
		$$
	\end{theorem}
	
	\begin{proof}
		Since $|V(G_p)|=p^3+p^2-p-1$, Theorem~\ref{thm:exactenergyodd} gives
		$$
		\mathcal E(G_p)-2\bigl(|V(G_p)|-1\bigr)
		=\sqrt{\Delta_p}+p\sqrt{\Theta_p}-4p^2+p+3,
		$$
		where $\Delta_p=4p^4-4p^2-8p+9$ and $\Theta_p=p^4-8p+8$. For $p\ge3$,
		$$
		\Delta_p-(2p^2-2)^2=4(p-1)^2+1>0,
		$$
		so $\sqrt{\Delta_p}>2p^2-2$. Also
		$$
		\Theta_p-4p^2=p^2(p^2-4)-8p+8\ge5p^2-8p+8>0,
		$$
		and therefore $p\sqrt{\Theta_p}>2p^2$. It follows that
		$$
		\mathcal E(G_p)-2\bigl(|V(G_p)|-1\bigr)>p+1>0.
		$$
		Thus every odd-prime graph in the family is hyperenergetic.
	\end{proof}
	
	\medskip
	The preceding reduction also converts the equitable quotient from Proposition~\ref{prop:equitable} into an energy lower bound.
	
	\begin{corollary}\label{cor:energyquotient}
		For every odd prime $p$,
		$$
		\mathcal E(G_p)\ge (p+1)(p-2)+\sqrt{4p^4-4p^2-8p+9}.
		$$
	\end{corollary}
	
	\begin{proof}
		By Proposition~\ref{prop:equitable}, the quotient matrix
		$$
		Q_p=
		\begin{pmatrix}
			p-2 & 2p(p-1)\\[1mm]
			2(p-1) & (2p-1)(p-1)
		\end{pmatrix}
		$$
		has eigenvalues
		$$
		\lambda_{\pm}
		=
		p^2-p-\frac12
		\pm
		\frac12\sqrt{4p^4-4p^2-8p+9}.
		$$
		The corresponding lifted quotient eigenvectors are constant on the type classes, so they lie in the class-constant subspace $U$. Hence $\lambda_+$ and $\lambda_-$ are eigenvalues of the reduced matrix $B_p$. Since $\lambda_+>0>\lambda_-$,
		$$
		|\lambda_+|+|\lambda_-|
		=
		\lambda_+-\lambda_-
		=
		\sqrt{4p^4-4p^2-8p+9}.
		$$
		Combining this with Theorem~\ref{thm:energyoddprime} gives the required inequality.
	\end{proof}
	
	\medskip
	A second lower bound follows from the spectral second moment of the reduced
	matrix.
	
	\begin{proposition}\label{prop:energymoment}
		For every odd prime $p$,
		$$
		\tr(B_p^2)=2|E(G_p)|-(p+1)(p-2)
		=(p+1)(2p^4-p^3-3p^2-p+4).
		$$
	\end{proposition}
	
	\begin{proof}
		By Theorem~\ref{thm:charfactor}, the spectrum of $A_p$ consists of the spectrum of $B_p$, together with $p(p+1)(p-2)$ zeros and $(p+1)(p-2)$ copies of $-1$. Therefore
		$$
		\tr(A_p^2)=\tr(B_p^2)+(p+1)(p-2).
		$$
		Since $A_p$ is the adjacency matrix of a simple graph, $\tr(A_p^2)=2|E(G_p)|$. Using Corollary~\ref{cor:edges},
		$$
		\tr(B_p^2)
		=(p^2-1)(2p^3+p^2-2p-2)-(p+1)(p-2),
		$$
		which simplifies to
		$$
		\tr(B_p^2)=(p+1)(2p^4-p^3-3p^2-p+4).
		$$
	\end{proof}
	
	\medskip
	The following is the immediate consequence of Theorem \ref{thm:energyoddprime}.
	
	\begin{corollary}\label{cor:energymoment}
		Let $p$ be an odd prime. Then
		$$
		\mathcal E(G_p)\ge
		(p+1)(p-2)
		+
		\frac{(p+1)(2p^4-p^3-3p^2-p+4)}
		{p^2-p-\frac12+\frac12\sqrt{4p^4-4p^2-8p+9}}.
		$$
		The displayed lower bound equals $p^3+O(p^2)$ as $p\to\infty$; in particular, $\mathcal E(G_p)=\Omega(p^3)$.
	\end{corollary}
	
	\begin{proof}
		By Theorem~\ref{thm:energyoddprime}, we have
		$$
		\mathcal E(G_p)=(p+1)(p-2)+\mathcal E(B_p).
		$$
		By Theorem~\ref{thm:specradius} and Proposition~\ref{prop:equitable}, $\rho(A_p)$ is an eigenvalue of $B_p$, and we have
		$$
		\rho(B_p)=\rho(A_p)
		=
		p^2-p-\frac12+\frac12\sqrt{4p^4-4p^2-8p+9}.
		$$
		Applying the symmetric-matrix inequality in Proposition~\ref{prop:energybasic} to $B_p$, and using Proposition~\ref{prop:energymoment}, we obtain
		$$
		\mathcal E(B_p)\ge \frac{\tr(B_p^2)}{\rho(B_p)}
		=
		\frac{(p+1)(2p^4-p^3-3p^2-p+4)}
		{p^2-p-\frac12+\frac12\sqrt{4p^4-4p^2-8p+9}}.
		$$
		Now, with $(p+1)(p-2)$, we obtain the required bound. Expanding the denominator gives $\rho(A_p)=2p^2-2p-1+O(p^{-1})$, so the displayed lower bound is $p^3+O(p^2)$.
	\end{proof}
	
	\medskip
	Corollary~\ref{cor:energymoment} is asymptotically stronger than the quotient bound in Corollary~\ref{cor:energyquotient}: the two quotient eigenvalues contribute only order $p^2$, whereas the second-moment estimate shows that the reduced energy is already of order $p^3$.

	\medskip
	The base case $n=2$ is explicit, while Theorem~\ref{thm:clique2power} supplies a general two-adic lower bound. Because graph energy is not monotone under taking subgraphs, the clique is used indirectly through the spectral radius.	
	\begin{theorem}\label{thm:energytwopower}
		For the two-adic family, $\mathcal E(G_2)=10$, and for every $t\ge 2$,
		$$
		\mathcal E(G_{2^t})\ge 2^{\,4\lfloor t/2\rfloor+1}-4.
		$$
	\end{theorem}
	
	\begin{proof}
		By Theorem~\ref{thm:n2},	$\mathcal E(G_2)=3+2+1+1+1+1+1=10.$ Now, let $t\ge 2$. Then by Corollary~\ref{cor:2powerbounds}, we have
		$$
		\rho(A_{2^t})\ge c_t-1=2^{4\lfloor t/2\rfloor}-2.
		$$
		Now, Proposition~\ref{prop:energybasic} implies that
		$$
		\mathcal E(G_{2^t})\ge 2\rho(A_{2^t})
		\ge 2\bigl(2^{4\lfloor t/2\rfloor}-2\bigr)
		=2^{\,4\lfloor t/2\rfloor+1}-4.
		$$
	\end{proof}
	
	\begin{remark}
		Theorem~\ref{thm:energytwopower} deliberately passes through the spectral radius rather than comparing energies directly, because graph energy is not monotone under taking subgraphs in general \cite{li2012energy,nikiforov2016energy}. The clique block forces a large spectral radius, and Proposition~\ref{prop:energybasic} converts that information into an energy lower bound.
	\end{remark}

	For compact notation in the Tables \ref{tab:energyodd} and \ref{tab:energy2power}, define
	$$
	L_p^{(Q)}:=(p+1)(p-2)+\sqrt{4p^4-4p^2-8p+9},
	$$
	and
	$$
	L_p^{(M)}:=
	(p+1)(p-2)
	+
	\frac{(p+1)(2p^4-p^3-3p^2-p+4)}
	{p^2-p-\frac12+\frac12\sqrt{4p^4-4p^2-8p+9}}.
	$$
	Tables~\ref{tab:energyodd} and \ref{tab:energy2power} compare the resulting bounds with small numerical cases. All decimal values in these tables and in Examples~\ref{ex:energyp3} and \ref{ex:energy4} were computed with Python 3 and NumPy 2.3.5 in double precision. For $p=3,5,7$, the symmetric reduced matrix $B_p=(p-1)H_p-D_p$ was also formed directly from the type-incidence rule and diagonalized numerically as an independent check. The resulting values agree with the closed formula in Theorem~\ref{thm:exactenergyodd}. For $n=4$, all $4^4=256$ quaternion elements were enumerated, the $127$ nonzero zero divisors were retained, and the $127\times127$ adjacency matrix was formed using the test $xy=0$ or $yx=0$ before symmetric diagonalization. The reported decimals are rounded to four places.
	
	\begin{table}[H]
		\centering
		\caption{Energy data for odd primes. The last column is the numerical evaluation of the exact formula in Theorem~\ref{thm:exactenergyodd}; direct diagonalization of $B_p$ gives the same values.}
		\label{tab:energyodd}
		\begin{tabular}{cccccc}
			\toprule
			$p$ & $|V(G_p)|$ & forced part $(p+1)(p-2)$ & $L_p^{(Q)}$ & $L_p^{(M)}$ & numerical $\mathcal E(G_p)$ \\
			\midrule
			$3$ & $32$  & $4$  & $20.5227$  & $35.6829$  & $72.7095$ \\
			$5$ & $144$ & $18$ & $66.6724$  & $161.5800$ & $364.4303$ \\
			$7$ & $384$ & $40$ & $136.7523$ & $423.5501$ & $1016.3064$ \\
			\bottomrule
		\end{tabular}
	\end{table}
	
	\begin{table}[H]
		\centering
		\caption{Energy values and lower bounds in the two-adic family.}
		\label{tab:energy2power}
		\small
		\begin{tabular}{cccc}
			\toprule
			$t$ & $n=2^t$ & $|V(G_{2^t})|$ & energy information \\
			\midrule
			$1$ & $2$  & $7$     & $\mathcal E(G_2)=10$ \\
			$2$ & $4$  & $127$   & $\mathcal E(G_4)\ge 28$; numerically $102.8092$ \\
			$3$ & $8$  & $2047$  & $\mathcal E(G_8)\ge 28$ \\
			$4$ & $16$ & $32767$ & $\mathcal E(G_{16})\ge 508$ \\
			\bottomrule
		\end{tabular}
	\end{table}
	
	\begin{example}\label{ex:energyp3}
		For $p=3$, Theorem~\ref{thm:exactenergyodd} gives $\mathcal E(G_3)=32+\sqrt{273}+3\sqrt{65}\approx72.7095.$ Equivalently, Theorem~\ref{thm:energyoddprime} gives $\mathcal E(G_3)=4+\mathcal E(B_3)$. The quotient lower bound from Corollary~\ref{cor:energyquotient} is $L_3^{(Q)}=4+\sqrt{273}\approx 20.5227,$ 	while the stronger moment lower bound from
		Corollary~\ref{cor:energymoment} is $L_3^{(M)}\approx 35.6829.$ Since $|V(G_3)|=32$ and the complete graph $K_{32}$ has energy $\mathcal E(K_{32})=2(32-1)=62,$ the graph $G_3$ is hyperenergetic (when energy exceeds that of a complete graph with the same number of vertices).
	\end{example}
	\begin{example}\label{ex:energy4}
		For $n=4$, direct construction of the $127\times 127$ adjacency matrix
		shows that $\rho(A_4)\approx 22.8577,$ and $\mathcal E(G_4)\approx 102.8092.$  Thus, the general lower bound
		$\mathcal E(G_4)\ge 28$ is nontrivial but far from sharp. This indicates that, within the two-adic family, the clique block encompasses merely a segment of the spectrum, and a more refined block decomposition remains absent.
	\end{example}
	
	The energy data indicate a significant disparity between the two regimes. In the odd-prime family, Theorem~\ref{thm:hyperodd} shows that hyperenergeticity holds for every odd prime, not only for the displayed numerical cases. In the two-adic family, the exact base graph $G_2$ is not hyperenergetic, and the available general estimates still come from forced clique and annihilator-layer structure rather than a complete spectral reduction. This further illustrates the structural distinction between the matrix-ring model $\LL_p\cong M_2(\F_p)$ and the more singular behavior of $\LL_{2^t}$.

	\section{Conclusion}\label{con}
	The paper develops an adjacency-matrix description of zero-divisor graphs arising from Lipschitz quaternion rings modulo $n$. For an odd prime $p$, the kernel--image classification of nonzero singular matrices in $M_2(\F_p)$ yields an explicit block model for $A_p$. Besides the degree and edge formulas, the block geometry gives the exact clique number $p+1$ and the exact independence number $(p-1)(p+1)^2/4$. The class-constant reduction first isolates the large forced $0$- and $-1$-eigenspaces, and a second symmetry decomposition diagonalizes the reduced matrix $B_p$ completely. As a result, the full characteristic polynomial, nullity, rank, spectral radius, and energy are explicit. In particular, every odd-prime graph $G_p$ is hyperenergetic.
	
	For the two-adic family, the graph at $n=2$ is completely determined and is isomorphic to the friendship graph $F_3$. For $t\ge2$, the ideal $2^{\lceil t/2\rceil}\LL_{2^t}$ is square-zero, and its nonzero elements form a clique of explicitly known size. The nested ideals $2^a\LL_{2^t}$ provide a second mechanism: when $a+b\ge t$, disjoint ideal layers are completely joined and yield large complete bipartite subgraphs. These constructions give spectral-radius and edge bounds, while the clique principal block also forces negative inertia and rank. Unlike the odd-prime case, however, a full block diagonalization of $A_{2^t}$ is still not available.
	
	Several directions remain open. The odd-prime calculation suggests that ring-theoretic strata should be identified before attempting numerical diagonalization, but prime powers and composite moduli introduce additional annihilator layers that are not visible over a field.
	
	We conclude with three separate questions.
	\begin{openproblem}
		Let $n$ be odd and squarefree with at least two prime factors. Use the Chinese remainder decomposition of $\LL_n$ to determine a block model for $A(G_n)$ and decide whether its characteristic polynomial can be expressed explicitly in terms of the prime components.
	\end{openproblem}
	
	\begin{openproblem}
		Determine the full adjacency spectrum of $\Gamma(\LL_{p^t})$ for odd primes $p$ and integers $t\ge2$.
	\end{openproblem}
	
	\begin{openproblem}
		For $t\ge2$, determine the exact clique number and the exact characteristic polynomial of $A_{2^t}$. The clique in Theorem~\ref{thm:clique2power} gives a general lower bound, but its maximality is not settled here.
	\end{openproblem}
	
	The block and quotient methods used above may also be useful for other zero-divisor graphs of finite noncommutative rings.
	\section*{Declarations}
	\noindent \textbf{Data Availability:}	There is no data associated with this article as data sharing is not applicable since no data sets were generated or analyzed during the current study.
	
	\noindent \textbf{Funding:} The authors declare that no funds, grants, or other support were received during the preparation of this manuscript.
	
	\noindent \textbf{Conflict of interest:} The authors have no competing interests to declare that are relevant to the content of this article.
	
	
	\noindent\textbf{\large Acknowledgement}
	All numerical evaluations were performed in Python 3 utilizing the NumPy library (v2.3.5). Figure~\ref{fig:G3-joined-union-manual} was generated via TikZ, with the underlying code produced using online AI-assisted tools.

\end{document}